\documentclass[12pt]{article}

\usepackage{amsmath,amsthm,amsfonts,amssymb,amscd,cite}
\usepackage{graphicx}

\allowdisplaybreaks[4]

\newtheorem {lemma}{Lemma}[section]
\newtheorem {theorem} {Theorem}[section]
\newtheorem {conjecture}{Conjecture}[section]
\newtheorem {corollary}{Corollary}[section]

\usepackage[figurename=Fig.]{caption}
\allowdisplaybreaks [4]

\usepackage{geometry} 

\usepackage{tikz}
\usetikzlibrary{positioning}
\usepackage{xcolor}
\usepackage{circuitikz}
\usetikzlibrary{patterns}
\usetikzlibrary{shapes.geometric}

\begin{document}

\title{\bf Diameter vs Laplacian eigenvalue distribution}

\author{Leyou Xu\footnote{Email: leyouxu@m.scnu.edu.cn}, Bo Zhou\footnote{Corresponding author; Email: zhoubo@m.scnu.edu.cn}\\
School of Mathematical Sciences, South China Normal University\\
Guangzhou 510631, P.R. China}

\date{}
\maketitle

\begin{abstract}
Let $G$ be a simple graph of order $n$.
It is known that any Laplacian eigenvalue of $G$ belongs to the interval $[0,n]$.
For an interval $I\subseteq [0, n]$, denote by $m_GI$ the number of Laplacian eigenvalues of $G$ in $I$, counted with multiplicities.
Let $G$ be a graph of order $n$ with diameter $d$.
In this paper, we show that
$m_G[n-d,n]\le n-d+2$ if $2\le d\le n-4$, and
 $m_G[n-2d+4,n]\le n-3$ if $3\le d\le \lfloor\frac{n+1}{2}\rfloor$. The diameter constraint provides an insightful
approach to understand how the Laplacian eigenvalues are distributed.
 \\ \\
{\bf  Keywords:} diameter, Laplacian spectrum,  permutational similar\\ \\
{\bf MSC(2020):} 05C12, 05C50, 15A18\\ \\
\end{abstract}


\section{Introduction}

We consider simple graphs.
%
Let $G$ be a graph with vertex set $\{v_1,\dots, v_n\}$.  The adjacency matrix $A(G)$ of $G$ is the $n\times n$ matrix where the $(i,j)$-entry is equal to  $1$ if $v_i$ and $v_j$ are adjacent, and is otherwise equal to $0$.
 Moreover, if $D(G)$ is the $n\times n$ diagonal matrix whose $(i,i)$-entry is the degree of vertex $v_i$ for $i=1,\dots, n$,  then $L(G)=D(G)-A(G)$ is called the Laplacian matrix  of $G$. This  is a symmetric positive semidefinite matrix and hence has $n$ real
nonnegative eigenvalues, which are said to be the Laplacian eigenvalues of $G$ and  can be arranged as
\[
\mu_n(G)\le \dots\le \mu_1(G),
\]
counted with multiplicities.
One can see that $\mu_n(G)=0$, and $\mu_j(G)$ is the $j$th (largest) Laplacian eigenvalue  of $G$ for $j=1,\dots, n$.

For  a graph $G$ of order $n$, any Laplacian eigenvalue of $G$ lies in the interval $[0,n]$ \cite{Me,Moh, Mo}, and the multiplicity of the Laplacian eigenvalue $0$ is equal to the number of the (connected) components of $G$ \cite{F}.
The distribution of Laplacian eigenvalues of graphs 
is relevant to the many applications related to Laplacian matrices \cite{GM,Me,Mo}
and there are results on the Laplacian eigenvalues in subintervals of $[0, n]$, see, e.g., \cite{AhMS,BRT,CMP,CJT,COP,GMS,GWZF,GT,HJT,JOT,Me2,Sin,ZZD}.
However, it is not well understood how the Laplacian eigenvalues are distributed in $[0, n]$, see \cite{JOT}.

The diameter of a connected graph $G$ is defined as the maximum distance
over all pairs of vertices in $G$.
Recently, progress is made on the connections between the distribution of the Laplacian eienvalues and the diameter. For an interval $I\subseteq [0,n]$, denote by $m_GI$ the number of Laplacian eigenvalues of a graph $G$ of order $n$ in $I$, counted with multiplicities.
The first result (Theorem \ref{B1}) was achieved by Ahanjideh,  Akbari, Fakharan and  Trevisan \cite{AhMS},  the second result (Theorem \ref{x1}) was conjectured in \cite{AhMS} and confirmed in \cite{XZ1}, and the third one (Theorem \ref{x2}) came from \cite{XZ}.

\begin{theorem}\cite{AhMS} \label{B1}
Let $G$ be a connected graph of order $n$ with diameter $d$, where $d\ge 4$. Then $m_G(n-d+3,n]\le n-d-1$.
\end{theorem}

\begin{theorem}\label{x1}\cite{AhMS,XZ1}
Let $G$ be a connected graph of order $n$ with diameter $d$, where $2\le d\le n-2$. Then  $m_G[n-d+2,n]\le n-d$.
\end{theorem}

\begin{theorem}\label{x2}\cite{XZ}
Let $G$ be a connected graph of order $n$ with diameter $d$, where $1\le d\le n-3$. Then $m_G[n-d+1,n]\le n-d+1$.
\end{theorem}

%
%
%

The first result of this paper is as follows.

\begin{theorem}\label{x}
Let $G$ be a connected graph of order $n$ with diameter $d$, where $2\le d\le n-4$.

(i) If $d=2,3,4$, then $m_G[n-d,n]\le n-d+1$.

(ii)  If $d\ge 5$, then $m_G[n-d,n]\le n-d+2$.
\end{theorem}

We  remark that the diameter condition in Theorem \ref{x} is tight.  It is known \cite{AM} that
$\mu_j(P_n)=4\sin^2\frac{(n-j)\pi}{2n}$ for $j=1,\dots, n$. Thus we have:

(i) If $d=n-1$, then $G$ is a path $P_n$, and $\mu_j(P_n)\ge 1$ if and only if $j=1,\dots,
\lfloor\frac{2}{3}n\rfloor$, so $m_{G}[1,n]=\lfloor\frac{2}{3}n\rfloor$.

(ii)  If $d=n-2$, then $P_{n-1}$ is a subgraph of $G$, and $\mu_j(P_{n-1})\ge 2$ if and only if $j=1,\dots,
\lfloor\frac{1}{2}(n-1)\rfloor$, so we have by Lemma \ref{cauchy} that
$\mu_{\lfloor\frac{1}{2}(n-1)\rfloor}(G)\ge \mu_{\lfloor\frac{1}{2}(n-1)\rfloor}(P_{n-1})\ge 2$, implying $m_G[2,n]\ge \lfloor\frac{1}{2}(n-1)\rfloor$.

(iii) If $d=n-3$, then  $P_{n-2}$ is a subgraph of $G$, and
 $\mu_j(P_{n-2})\ge 3$ if and only if $j=1,\dots,
\lfloor\frac{1}{3}(n-2)\rfloor$, so we have
by Lemma \ref{cauchy} that  $m_G[3,n]\ge \lfloor\frac{1}{3}(n-2)\rfloor$.


Next, we give the second result.

\begin{theorem}\label{d-2}
Let $G$ be a connected graph of order $n$ with diameter $d$.

(i) If $d=2, \frac{n+2}{2},\frac{n+3}{2}$, then $m_G[n-2d+4,n]\le n-2$.

(ii) If $3\le d\le \lfloor\frac{n+1}{2}\rfloor$, then $m_G[n-2d+4,n]\le n-3$.
\end{theorem}

Note that the case $d=3$ in Theorem \ref{d-2} (ii) has been given in \cite{XZ}. Also,  the diameter condition is tight in Theorem \ref{d-2}, see Section 4.

%
%

Suppose that $G$ is a connected graph of order $n$  with diameter $d\ge 2$.
Motivated by Theorems \ref{x1}--\ref{d-2} and the trivial fact that
$m_G[n-2d+3,n]\le n-1$
if $d\le \lfloor \frac{n-1}{2}\rfloor$ and $m_G[n-2d+2,n]\le n$ if $d\le \lfloor \frac{n-2}{2}\rfloor$,
we are tempted  to conjecture that  if $c=0, \dots, d$ with  $\max\{2,c\}\le d\le n-2-c$, then $m_G[n-d+2-c,n]\le n-d+c$.

\section{Preliminaries}

Let $G$ be a graph of order $n$  with vertex set $V(G)$ and edge set $E(G)$. For a vertex $v$ of $G$, the neighborhood of $v$, denoted by $N_G(v)$,  is the set of vertices that are adjacent to $v$ in $G$, and
the degree of $v$, denoted by $\delta_G(v)$, is the number of vertices that are adjacent to $v$ in $G$, i.e., $\delta_G(v)=|N_G(v)|$.
The degree sequence of $G$ is the sequence $(\delta_1(G),\dots, \delta_n(G))$ of the degrees of the vertices in non-increasing order.
For $S\subseteq V(G)$, denote by $G[S]$ the subgraph of $G$ induced by $S$ if $S\ne \emptyset$ and $G-S$ denotes $G[V(G)\setminus S]$, that is, the subgraph obtained from $G$ by deleting the vertices of $S$ if $S\ne V(G)$.
For  $F\subseteq E(G)$, denote by $G-F$ the subgraph of $G$ obtained from $G$ by deleting all edges in $F$. In particular, if $F=\{e\}$, then we write   $G-e$ for $G-\{e\}$.
Now suppose that $G$ is connected. The distance between
vertices $v,w$, denoted by $d_G(v, w)$, is the length of  a shortest path
between $v$ and $w$ in $G$.  The diameter
of $G$ is $\max\{d_G(v, w) : v, w \in V (G)\}$.
A path of $G$ that joins a pair of vertices whose  distance is equal to the diameter  is called a diametral path.

For vertex disjoint graphs $G$ and $H$,  denote by $G\cup H$  the disjoint union of them.  The disjoint union of $k$ copies of $G$ is denoted by $kG$.
Denote by $P_n$ the path of order $n$ and $K_n$ the complete graph of order $n$.
For undefined notation and terminology we refer to  \cite{BM}.


%

We need the following lemmas in our proofs.

For an $n\times n$ Hermitian matrix $M$, $\rho_i(M)$ denotes its $i$-th largest eigenvalue of $M$. 
%
We need Weyl's inequalities \cite{We,KT}
with a characterization  of the equality cases
\cite[Theorem 1.3]{So}.

\begin{lemma} \cite[Theorem 1.3]{So} \label{cw}
Let $A$ and $B$ be Hermitian matrices of order $n$.
For $1\le i,j\le n$ with $i+j-1\le n$,
\[
\rho_{i+j-1}(A+B)\le \rho_i(A)+\rho_j(B).
 \]
with equality if and only if there exists a nonzero vector $\mathbf{x}$ such that $\rho_{i+j-1}(A+B)\mathbf{x}=(A+B)\mathbf{x}$, $\rho_i(A)\mathbf{x}=A\mathbf{x}$ and $\rho_j(B)\mathbf{x}=B\mathbf{x}$.
\end{lemma}

We also need two types of interlacing theorem or inclusion principle.

\begin{lemma}\label{interlacing}\cite[Theorem 4.3.28]{HJ}
If $M$ is  a Hermitian matrix of order $n$ and $B$ is its principal submatrix of order $p$, then $\rho_{n-p+i}(M)\le\rho_i(B)\le \rho_{i}(M)$ for $i=1,\dots,p$.
\end{lemma}

\begin{lemma}\label{cauchy} \cite[Theorem 3.2]{Moh}
If $G$ is a graph on $n$ vertices with $e\in E(G)$, then
\[
\mu_1(G)\ge\mu_1(G-e)\ge \mu_2(G)\ge\dots\ge \mu_{n-1}(G-e)\ge \mu_n(G)=\mu_{n}(G-e)=0.
\]
\end{lemma}

For integers $n$, $d$ and $t$ with $2\le d\le n-2$ and $2\le t\le d$, let $P_{d+1}:=u_1\dots u_{d+1}$,  $V=V(K_{n-d-1})$ and let $G_{n,d,t}$ be the graph  obtained from the disjoint union of $P_{d+1}$ and  $K_{n-d-1}$  by adding all   edges in $\{u_iw: i=t-1,t,t+1, w\in V\}$.

%

\begin{lemma}\label{gndt}\cite{XZ1} For integers $n$, $d$ and $t$ with $2\le d\le n-2$ and $2\le t\le d$, $\mu_{n-d}( G_{n,d,t})=n-d+2$.
\end{lemma}

For integers $n$, $p$ and $q$ with $2\le p\le q\le n-3$, let $H_{n,p,q}$ be the graph obtained from
$G_{n-1, n-3,p}$ (with a diametral path $u_1\dots u_{n-2}$ and an additional vertex $u$ outside) by adding a vertex $v$ and three edges connecting $v$ and $u_{q-1}, u_q$ and $u_{q+1}$ if $q\ge p+2$
and four edges connecting $v$ and $u_{q-1}, u_q$, $u_{q+1}$ and $u$ if $q=p,p+1$, see Figs.~\ref{PQQ} and \ref{PQ}.

\begin{figure}[htbp]
\centering
\begin{tikzpicture}
\draw  [black](0,0)--(0.5,0);
\draw  [black, dashed](0.5,0)--(1.5,0);
\draw  [black](1.5,0)--(2,0)--(3,0)--(4,0)--(4.5,0);
\filldraw [black] (0,0) circle (2pt);
\filldraw [black] (2,0) circle (2pt);
\filldraw [black] (3,0) circle (2pt);
\filldraw [black] (4,0) circle (2pt);
\draw  [black, dashed](4.5,0)--(5.5,0);
\draw  [black](5.5,0)--(6,0)--(7,0)--(8,0)--(8.5,0);
\filldraw [black] (6,0) circle (2pt);
\filldraw [black] (7,0) circle (2pt);
\filldraw [black] (8,0) circle (2pt);
\draw  [black, dashed](8.5,0)--(9.5,0);
\draw  [black](9.5,0)--(10,0);
\filldraw [black] (10,0) circle (2pt);
\filldraw [black] (3,2) circle (2pt);
\filldraw [black] (7,2) circle (2pt);
\draw  [black](2,0)--(3,2);
\draw  [black](3,0)--(3,2);
\draw  [black](4,0)--(3,2);
\draw  [black](6,0)--(7,2);
\draw  [black](7,0)--(7,2);
\draw  [black](8,0)--(7,2);
\node at (0,-0.4) {$u_1$};
\node at (2, -0.4) {$u_{p-1}$};
\node at (3,-0.4) {$u_p$};
\node at (4,-0.4) {$u_{p+1}$};
\node at (6,-0.4) {$u_{q-1}$};
\node at (7, -0.4) {$u_q$};
\node at (8,-0.4) {$u_{q+1}$};
\node at (10,-0.4) {$u_{n-2}$};
\node at (3,2.3) {$u$};
\node at (7,2.3) {$v$};
\end{tikzpicture}
\caption{The graph $H_{n,p,q}$ with $q\ge p+2$.}
\label{PQQ}
\end{figure}
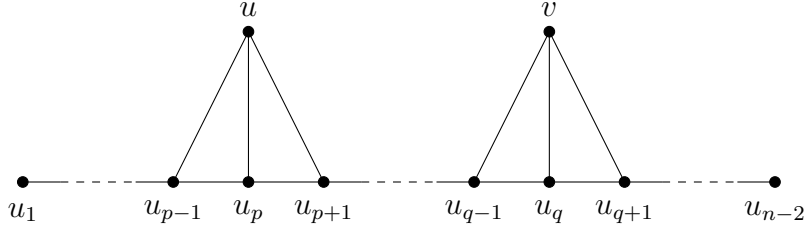


\begin{figure}[htbp]
\centering
\begin{tikzpicture}
\draw  [black](0,0)--(0.5,0);
\draw  [black, dashed](0.5,0)--(1.5,0);
\draw  [black](1.5,0)--(2,0)--(3,0)--(4,0)--(5,0)--(5.5,0);
\filldraw [black] (0,0) circle (2pt);
\filldraw [black] (2,0) circle (2pt);
\filldraw [black] (3,0) circle (2pt);
\filldraw [black] (4,0) circle (2pt);
\filldraw [black] (5,0) circle (2pt);
\draw  [black, dashed](5.5,0)--(6.5,0);
\draw  [black](6.5,0)--(7,0);
\filldraw [black] (7,0) circle (2pt);
\filldraw [black] (3,2) circle (2pt);
\filldraw [black] (4,2) circle (2pt);
\draw  [black](2,0)--(3,2);
\draw  [black](3,0)--(3,2);
\draw  [black](4,0)--(3,2);
\draw  [black](3,0)--(4,2);
\draw  [black](4,0)--(4,2);
\draw  [black](5,0)--(4,2);
\draw  [black](3,2)--(4,2);
\node at (0,-0.4) {$u_1$};
\node at (2, -0.4) {$u_{p-1}$};
\node at (3,-0.4) {$u_p$};
\node at (4,-0.4) {$u_{p+1}$};
\node at (5,-0.4) {$u_{p+2}$};
\node at (7,-0.4) {$u_{n-2}$};
\node at (3,2.3) {$u$};
\node at (4,2.3) {$v$};
\end{tikzpicture}
\begin{tikzpicture}
\draw  [black](0,0)--(0.5,0);
\draw  [black, dashed](0.5,0)--(1.5,0);
\draw  [black](1.5,0)--(2,0)--(3,0)--(4,0)--(4.5,0);
\filldraw [black] (0,0) circle (2pt);
\filldraw [black] (2,0) circle (2pt);
\filldraw [black] (3,0) circle (2pt);
\filldraw [black] (4,0) circle (2pt);
\draw  [black, dashed](4.5,0)--(5.5,0);
\draw  [black](5.5,0)--(6,0);
\filldraw [black] (6,0) circle (2pt);
\filldraw [black] (2.5,2) circle (2pt);
\filldraw [black] (3.5,2) circle (2pt);
\draw  [black](2,0)--(2.5,2);
\draw  [black](3,0)--(2.5,2);
\draw  [black](4,0)--(2.5,2);
\draw  [black](2,0)--(3.5,2);
\draw  [black](3,0)--(3.5,2);
\draw  [black](4,0)--(3.5,2);
\draw  [black](2.5,2)--(3.5,2);
\node at (0,-0.4) {$u_1$};
\node at (2, -0.4) {$u_{p-1}$};
\node at (3,-0.4) {$u_p$};
\node at (4,-0.4) {$u_{p+1}$};
\node at (6,-0.4) {$u_{n-2}$};
\node at (2.5,2.3) {$u$};
\node at (3.5,2.3) {$v$};
\end{tikzpicture}
\caption{The graph $H_{n,p,q}$ with $q= p+1$ (left) and $q=p$ (right).}
\label{PQ}
\end{figure}
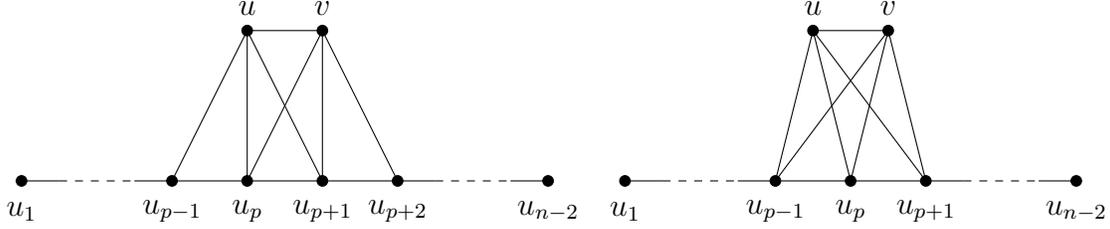

For integers $n$, $d$, $r$ and $a$ with $3\le d\le n-2$, $2\le r\le d-1$ and $1\le a\le n-d-2$,  let
$P_{d+1}:=u_1\dots u_{d+1}$  and $V(K_{n-d-1})=V_1\cup V_2$ with $|V_1|=a$, and let
$G_{n,d,r,a}$ be the graph obtained from the disjoint union of $P_{d+1}$ and  $K_{n-d-1}$ by adding
all edges in $\{u_iv: i=r-1,r,r+1, v\in V_1\}\cup \{u_jw: j=r, r+1,r+2, w\in V_2\}$.

\begin{lemma}\label{Hpq} \label{gndla} The following statements are true.

(i) $\mu_1(P_n)<4$.

(ii) $\mu_5(H_{n,p,q})<4$.

(iii)  
$\mu_5(G_{7,3,2,1})<4$.
\end{lemma}

\begin{proof}
Part (i) follows from the fact that $\mu_1(P_n)=4\sin^2\frac{(n-1)\pi}{2n}$ given in \cite[p.~145]{AM}.

Part (ii) follows from the proof of \cite[Theorem 4]{XZ1}.

Part (iii) follows from a direct calculation that $\mu_5(G_{7,3,2,1})=3.4048<4$.
\end{proof}


\begin{lemma}\label{z+} \cite{XZ1}
Let $G$ be a connected graph of order $n$ with diameter $d$,  where $2\le d\le n-2$.
Suppose that $G$ is not isomorphic to $G_{n,d,t}$ for $2\le t\le d$, and $G$ is not isomorphic to  $G_{n,d,r,a}$
for  $2\le r\le d-1$ and $1\le a\le n-d-2$.
Then 
$m_G[n-d+2,n]\le n-d-1$.
\end{lemma}

Given a graph $G$ with $V(G)=\{v_1,\dots, v_n\}$ and a vector $\mathbf{x}=(x_1,\dots, x_n)^\top$
 can be viewed as a function defined on $V(G)$  mapping $v_i$ to
$x_{v_i}$ i.e., $\mathbf{x}(v_i) = x_{v_i}=x_i$ for  $i=1,\dots, n$.

A pendant path $u_1\dots u_p$ of $G$ at $u_p$ is an induced path of $G$ with $\delta_G(u_1)=1$, $\delta_G(u_p)\ge 3$, and $\delta_G(u_i)=2$ for $i=2,\dots, p-1$ if $p\ge 3$.

\begin{lemma}\label{cha}
Let $P:=v_1\dots v_{\ell}$ be a pendant path of a graph $G$ at $v_{\ell}$.
If there is a vector $\mathbf{x}$ such that $L(G)\mathbf{x}=4\mathbf{x}$, then for $i=1,\dots,\ell$, $x_i=(-1)^{i-1}(2i-1)x_1$, where $x_i=x_{v_i}$.
\end{lemma}

\begin{proof}  We prove the statement  by induction on $i$. It is trivial for $i=1$.
From $L(G)\mathbf{x}=4\mathbf{x}$ at $v_1$, we have $x_1-x_2=4x_1$,
i.e., $x_2=-3x_1$, so the statement is true for $i=2$. Suppose that $2\le i\le \ell-1$ and $x_j=(-1)^{j-1}(2j-1)x_1$ for each $j\le i$.
From $L(G)\mathbf{x}=4\mathbf{x}$ at $v_{i}$, we have \[
2x_{i}-x_{i-1}-x_{i+1}=4x_{i},
\]
so
\[
x_{i+1}=-2x_i-x_{i-1}=(-1)^i(2i+1)x_1. \qedhere
\]
\end{proof}

Denote by $\overline{G}$ the complement of $G$.

\begin{lemma}\label{n-} \cite{Moh} Let $G$ be a graph of order $n$. Then
$\mu_i(G)+\mu_{n-i}(\overline{G})=n$ for $i=1,\dots,n-1$.
\end{lemma}

\begin{lemma}\label{Delta}\cite{GM,PH} Let $G$ be a graph of order $n$ with at least one edge. Then
$\mu_1(G)\ge \delta_1(G)+1$  with equality when $G$ is connected if and only if $\delta_1(G)=n-1$.
Moreover, if $G$ is connected  with $n\ge 3$, then $\mu_2(G)\ge \delta_2(G)$ with equality only if,
under reordering the vertices so that 
 $\delta_G(v_i)=\delta_i(G)$ for $i=1,\dots,n$, $G$ satisfies one of the following conditions:

(i) $v_1v_2\notin E(G)$ and $N_G(v_1)=N_G(v_2)$,

(ii) $v_1v_2\in E(G)$, $\delta_1(G)=\delta_2(G)=\frac{n}{2}$ and $N_G(v_1)\cap N_G(v_2)=\emptyset$.
\end{lemma}

From Lemmas \ref{n-} and \ref{Delta}, we have

\begin{corollary}\label{delta}
Let $G$ be a graph of order $n$ that is not a complete graph. Then $\mu_{n-2}(G)\le \delta_{n-1}(G)+1$ with equality if and only if, under under reordering the vertices so that
$\delta_G(v_i)=\delta_i(G)$ for $i=1,\dots,n$,
 $G$ satisfies one of the following conditions:

(i) $v_{n-1}v_n\in E(G)$ and $N_G(v_{n-1})\setminus\{v_n\}=N_G(v_n)\setminus\{v_{n-1}\}$,

(ii) $v_{n-1}v_n\notin E(G)$, $\delta_{n-1}(G)=\delta_n(G)=\frac{n-2}{2}$ and $N_G(v_{n-1})\cap N_G(v_n)=\emptyset$.
\end{corollary}

A semi-regular bipartite graph is a bipartite graph
in which vertices in the same partite set have the same degree.
For a semi-regular bipartite graph $F$, let $F^+=F+\{uv: N_F(u)=N_F(v), u,v\in V(F)\}$.

\begin{lemma}\label{mu1}\cite{YLT}
For a graph $G$,  $\mu_1(G)\le \max\{\delta_G(u)+\delta_G(v)-|N_G(u)\cap N_G(v)|:uv\in E(G)\}$ with equality if and only if  for some semi-regular bipartite graph  $F$, $G\cong F^+$.
\end{lemma}

Let $G$ be a graph of order $n$.
Denote by $\kappa(G)$ the connectivity of $G$. By the well-known Whitney's inequality, $\kappa(G)\le \delta_n(G)$.
For two vertex disjoint  graphs $G_1$ and $G_2$, their join  is the graph $G_1\cup G_2+\{uv: u\in V(G_1), v\in V(G_2)\}$.

\begin{lemma}\label{mun-1}\cite{F,KMNS}
Let $G$ be a connected graph of order $n$ that is not complete. Then $\mu_{n-1}(G)\le\kappa(G)$ with equality if and only if $G$ is a join of two graphs $G_1$ and $G_2$, where $G_1$ is a disconnected graph of order $n-\kappa(G)$ and $G_2$ is a graph of order $\kappa(G)$ with $\mu_{\kappa(G)-1}(G_2)\ge 2\kappa(G)-n$.
\end{lemma}

Let $G$ be a connected graph and $P$ be a diametral path of $G$. For vertex $z$ of $G$ outside $P$,   we denote by $\Gamma_{G,P}(z)$ the set of neighbors of $z$ on $P$, that is, $\Gamma_{G,P}(z)=N_G(z)\cap V(P)$.

We say two matrices $A$ and $B$  are permutational similar if $A=QBQ^\top$ for some
permutation matrix $Q$.

\section{Proof of Theorem \ref{x}}

Theorem \ref{x} follows from Theorems \ref{d23} and \ref{normal}.

\begin{theorem}\label{d23}
Let $G$ be a connected graph of order $n$ with diameter $d$, where $d\le n-4$. If $d=2,3,4$, then  $m_G[n-d,n]\le n-d+1$.
\end{theorem}
\begin{proof}
The result for $d=2$ is trivial as $\mu_n(G)=0$.

Suppose that $d=3$. Let $P:=v_1\dots v_4$ be the diametral path of $G$. Then $v_1v_3,v_1v_4\in E(\overline{G})$, so $\delta_1(\overline{G})\ge 2$. By Lemma \ref{Delta}, $\mu_1(\overline{G})> \delta_1(\overline{G})+1=3$ as $n\ge 4$. So  by Lemma \ref{n-}, $\mu_{n-1}(G)=n-\mu_1(\overline{G})<n-3$. Thus, $m_G[n-3,n]\le n-2$.

Now suppose  that $d=4$.
It suffices to show that $\mu_{n-2}(G)<n-4$, or $\mu_2(\overline{G})>4$ by Lemma \ref{n-}.

Let $P:=v_1v_2v_3v_4v_5$ be the diametral path of $G$. Then $\overline{G}[\{v_1,v_2,v_3,v_4,v_5\}]$ is $H_0$ in Fig. \ref{P5c}.

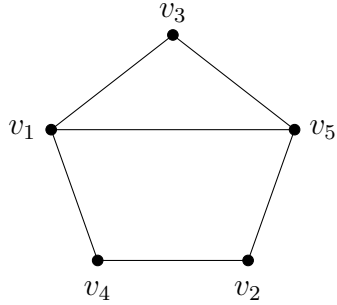
\begin{figure}[htbp]
\centering
\begin{tikzpicture}
\draw  [black](0,0)--(2,0);
\filldraw [black] (0,0) circle (2pt);
\filldraw [black] (2,0) circle (2pt);
\filldraw [black] (1,3) circle (2pt);
\filldraw [black] (-0.618,1.74) circle (2pt);
\filldraw [black] (2.618,1.74) circle (2pt);
\draw  [black](0,0)--(-0.618,1.74)--(1,3);
\draw  [black](2,0)--(2.618,1.74)--(1,3);
\draw  [black](-0.618,1.74)--(2.618,1.74);
\node at (0,-0.4) {$v_4$};
\node at (2, -0.4) {$v_2$};
\node at (-1,1.74) {$v_1$};
\node at (3,1.74) {$v_5$};
\node at (1,3.3) {$v_3$};
\end{tikzpicture}
\caption{The graph $H_0$.}
\label{P5c}
\end{figure}


As $n\ge 8$, there are at least three vertices outside $P$ in $G$.  Let  $u$, $v$ and $w$ be three such vertices.

Suppose first that $u$, $v$ and $w$ are all adjacent to $v_1$ in $G$. As  $P$ is a diametral path of $G$,  none of $u$, $v$ and $w$ is  adjacent to  $v_4$ or $v_5$ in $G$, so all of them are adjacent to both $v_4$ and $v_5$ in $\overline{G}$.   Thus  $\delta_1(\overline{G})\ge \delta_{\overline{G}}(v_5)\ge 6$ and $\delta_2(\overline{G})\ge \delta_{\overline{G}}(v_4)\ge 5$.
By Lemma \ref{Delta}, $\mu_2(\overline{G})\ge \delta_2(\overline{G})\ge 5>4$.

Suppose next that exactly two of $u$, $v$ and $w$, say $u$ and $v$, are adjacent to $v_1$ in $G$. Then $uv_5,vv_5\notin E(G)$, so $uv_5,vv_5,wv_1\in E(\overline{G})$, implying that  $\delta_{\overline{G}}(v_5)\ge 5$ and $\delta_{\overline{G}}(v_1)\ge 4$.
Thus $\delta_2(\overline{G})\ge 5$ and $\delta_2(\overline{G})\ge 4$.
By Lemma \ref{Delta}, $\mu_2(\overline{G})\ge \delta_2(\overline{G})\ge 4$. Suppose that $\mu_2(\overline{G})=4$. Then $\delta_{\overline{G}}(v_5)=\delta_1(\overline{G})\ge 5$ and
 $\delta_{\overline{G}}(v_1)=\delta_2(\overline{G})=4$.
 Note that $v_1$ and $v_5$ are adjacent in $\overline{G}$ with a common neighbor $v_3$. By Lemma \ref{Delta}, this is impossible. It thus follows that $\mu_2(\overline{G})> 4$.

Now, suppose that  exactly one of $u$, $v$ and $w$, say  $u$, is adjacent to $v_1$ in $G$. Then $uv_5\notin E(G)$, so $uv_5,vv_1,wv_1\in E(\overline{G})$. Thus
 $\delta_1(\overline{G})\ge \delta_{\overline{G}}(v_1)\ge 5$
and
  $\delta_2(\overline{G})\ge \delta_{\overline{G}}(v_5)\ge 4$.
As $v_1$ and $v_5$ are adjacent in $\overline{G}$ and with a common neighbor $v_3$,  one gets $\mu_2(\overline{G})>\delta_2(\overline{G})\ge 4$ by Lemma \ref{Delta}.

Finally, suppose  that none of $u$, $v$ and $w$ is adjacent to $v_1$ in $G$. If two of them, say $u$ and $v$,  are adjacent to $v_2$ in $G$,  then $uv_5,vv_5\in E(\overline{G})$, implying that $\delta_2(\overline{G})\ge 5$, so we have by Lemma \ref{Delta} that $\mu_2(\overline{G})>4$.
If  at most one of $u$, $v$ and $w$ is adjacent to $v_2$ in $G$,  then we may assume that $uv_2,vv_2\not\in E(G)$, i.e.,  $uv_2,vv_2\in E(\overline{G})$, so $\delta_2(\overline{G})\ge 4$ and
we have $\mu_2(\overline{G})>4$ by Lemma \ref{Delta}.
\end{proof}

It is evident that $m_{K_n-e}[n,n]=n-2$.
Note that  $G_{n,3,2}$ ($G_{n,4,3}$, respectively) is an $n$-vertex graph with diameter $3$ ($4$, respectively). As $K_{n-1}-e$ is a subgraph of $G_{n,3,2}$, we have $\mu_{n-2}(G_{n,3,2})\ge\mu_{n-2}(K_{n-1}-e)=n-3$, so by Theorem \ref{d23}, $m_{G_{n,3,2}}[n-3,n]=n-2$.
From \cite[Proposition 1]{XZ}, $\mu_{n-3}(G_{n,4,3})>n-3$, so by Theorem \ref{d23} again, $m_{G_{n,4,3}}[n-4,n]=n-3$. Thus,  the bound in Theorem \ref{d23} is tight.

\begin{theorem}\label{normal}
Let $G$ be a connected graph of order $n$ with diameter $d$, where $5\le d\le n-4$.  Then $m_G[n-d,n]\le n-d+2$.
\end{theorem}
\begin{proof}
Let $P:=v_1\dots v_{d+1}$ be a diametral path of $G$. As $d\le n-4$, there are at least three vertices lying outside $P$. Assume that  $u$, $v$ and $w$ are three such vertices.
Let $G'$ be the subgraph of $G$ induced by $V(P)\cup \{u,v,w\}$ and $B$ the principal submatrix of $L(G)$ corresponding to vertices of $G'$. Denote by $M$ the  diagonal matrix whose diagonal entry corresponding to vertex $z$ is $\delta_G(z) -\delta_{G'}(z)$ for $z\in V(G')$. Then  $B=L(G')+M$.
By  Lemma \ref{interlacing},
\[
\mu_{n-d+3}(G) =\rho_{n-(d+4)+7}(L(G))\le \rho_7(B).
\]
By Lemma \ref{cw},
\[
\rho_7(B) \le \mu_7(G')+  \rho_1(M).
\]
Thus  $\mu_{n-d+3}(G)\le \mu_7(G')+  \rho_1(M)$. Obviously, $\rho_1(M)\le n-|V(G')| = n-d-4$.
If $\mu_7(G')<4$, then $\mu_{n-d+3}(G)<n-d$,
so $m_G[n-d,n]\le n-d+2$.
Thus, it suffices to show that $\mu_7(G')<4$.

As $P$ is a diametral path of $G$, any vertex outside $P$ has at most three consecutive neighbors on $P$. For $z=u,v,w$, let $n_z=|\Gamma_{G,P}(z)|$. If $n_z<3$ for some $z=u,v,w$, then there exist $3-n_z$ vertices on $P$ so that $P$ remains to be a diametral path of the graph obtained by adding edges between $z$ and the $3-n_z$ vertices.
By Lemma \ref{cauchy}, we can assume that $\Gamma_{G,P}(u)= \{v_{p-1},v_p,v_{p+1}\}$, $\Gamma_{G,P}(v)= \{v_{q-1},v_q,v_{q+1}\}$ and $\Gamma_{G,P}(w)= \{v_{r-1},v_r,v_{r+1}\}$, where $2\le p,q,r\le d$.
Assume that $p\le q\le r$.
If $q-p>r-q$, then we relabel the vertices of $G$  by setting
$v'_i=v_{d+2-i}$ for $i=1,\dots, d+1$, $u'=w$, $v'=v$, and  $w'=w$, so
we have $p'\le q'\le r'$ and $q'-p'\le r'-q'$, where
$p'=d+2-r$, $q'=d+2-q$ and $r'=d+2-p$.  So,  we assume furthermore  that $q-p\le r-q$.

\noindent
{\bf Case 1.} $r\ge q+2$.

Note that $uw, vw\not\in E(G)$. It is easy to see that $G'-\{v_{r-1}v_r, wu_{r+1}\}\cong H_{d+4,p,q}$ or $H_{d+4,p,q}-uv$, where $u$ and $v$ are the two vertices outside the diametral path $P$.
By Lemmas \ref{cauchy} and \ref{Hpq}, one gets
\[
\mu_7(G')\le \mu_5(H_{d+4,p,q})<4,
\]
as desired.

\noindent
{\bf Case 2.} $r=q+1$.

By assumption, we have $p\le q\le p+1$.

\noindent
{\bf Case 2.1.} $q=p+1$.

It is possible that $v$ is adjacent to $u$ or $w$. Assume that $uv,vw\in E(G)$ by Lemma \ref{cauchy}.
Let $u_i=v_i$ for $i=1,\dots,p-1$, $u_p=u$, $u_{i+1}=v_i$ for $i=p,p+1$, $u_{p+3}=w$, $u_{i+2}=v_i$ for $i=p+2,\dots,d+1$ and $u_{d+4}=v$. Under this  new labeling,
\[
G'-\{u_{p-1}u_{p+1}, u_pu_{p+2}, u_{p+2}u_{p+4}, u_{p+3}u_{p+5}, u_{p}u_{d+4}, u_{p+4}u_{d+4}\}
 \]
is a copy of $G_{d+4,d+2,p+2}$.
So
\[
L(G')=L(G_{d+4,d+2,p+2})+R,
\]
where $R=(r_{ij})_{(d+4)\times (d+4)}$ with \[
r_{ij}=\begin{cases}
1&\mbox{ if }i=j\in \{p-1,p+1,p+3,p+5 \},\\
2&\mbox{ if }i=j\in \{p,p+2,p+4,d+4 \},\\
-1&\mbox{ if }\{i,j\}\in \{\{p-1,p+1\},\{p,p+2\},\{p,d+4\} \},\\
-1&\mbox{ if }\{i,j \}\in \{\{p+2,p+4\},\{p+3,p+5\},\{ p+4,d+4\} \},\\
0&\mbox{ otherwise.}
\end{cases}
\]
As $R$ is permutational similar to  $L(2P_2\cup C_4\cup (d-4)K_1)$, 
we have $\rho_6(R)=0$. So by Lemmas \ref{cw} and \ref{gndt}, we have \[
\mu_7(G')\le \mu_2(G_{d+4,d+2,p+2})+\rho_6(R)=4.
\]

Suppose that $\mu_7(G')=4$. By Lemma \ref{cw}, there exists a nonzero vector $\mathbf{x}$ such that $R\mathbf{x}=\mathbf{0}$ and $L(G_{d+4,d+2,p+2})\mathbf{x}=4\mathbf{x}$.
Let $x_i=x_{u_i}$ for $i=1,\dots,d+4$.

From $R\mathbf{x}=\mathbf{0}$, we have
$L(C_4)(x_p,x_{p+2}, x_{p+4}, x_{d+4})^\top=\mathbf{0}$, so
$x_p=x_{p+2}=x_{p+4}=x_{d+4}$.

From $R\mathbf{x}=\mathbf{0}$ at $u_{p-1}$ and $u_{p+3}$, respectively, we have $x_{p-1}=x_{p+1}$ and $x_{p+3}=x_{p+5}$.

From $L(G_{d+4,d+2,p+2})\mathbf{x}=4\mathbf{x}$ at $u_p$, we have \[
2x_p-x_{p-1}-x_{p+1}=4x_p,
\]
so $x_{p-1}=-x_p$.

%
%
As $u_1\dots u_{p+1}$ is a pendant path of $G_{d+4,d+2,p+2}$ at $u_{p+1}$, we have by Lemma \ref{cha}
that
\[
x_i=(-1)^{i-1}(2i-1)x_1 \mbox{ for }i=1,\dots,p+1.
\]
From $L(G_{d+4,d+2,p+2})\mathbf{x}=4\mathbf{x}$ at $u_{p+1}$, we have
 $3x_{p+1}-x_{p}-x_{p+2}-x_{d+4}=4x_{p+1}$,  so $x_{p+1}=-3x_{p}$.
It hence follows that \[
(-1)^{p}(2p+1)x_1=x_{p+1}=-3x_{p}=-3(-1)^{p-1}(2p-1)x_1,
\]
i.e., \[
(2p+1)x_1=3(2p-1)x_1,
\]
equivalently,  $x_1=0$. So $x_i=0$ for $i=1,\dots,p+2,p+4,d+4$.
From $L(G_{d+4,d+2,p+2})\mathbf{x}=4\mathbf{x}$ at $u_{d+4}$, we have \[
3x_{d+4}-x_{p+1}-x_{p+2}-x_{p+3}=4x_{d+4}.
\]
As $x_{d+4}=x_{p+1}=x_{p+2}=x_{p+4}=0$, one gets $x_{p+3}=0$. So $x_{p+5}=x_{p+3}=0$. It follows that
$x_i=0$ for $i=1,\dots,p+5$. Now from $L(G_{d+4,d+2,p+2})\mathbf{x}=4\mathbf{x}$ at $u_i$ for $i=p+5,\dots, d+2$, we have $x_{i+1}=0$.
Thus $\mathbf{x}$ is a zero vector, a contradiction.
Therefore, $\mu_7(G')<4$.

\noindent
{\bf Case 2.2.} $q=p$.

By Lemma \ref{cauchy}, we assume that $uv,vw,uw\in E(G')$.

If $3\le p\le d-2$, then $G'-\{v_{p-2}v_{p-1}, v_{p+2}v_{p+3}\}\cong G_{7,3,2,1}\cup P_{p-2}\cup P_{d-p-1}$, so we have
by Lemmas \ref{cauchy} and \ref{Hpq} that
\[
\mu_7(G')\le \mu_5(G'-\{v_{p-2}v_{p-1},v_{p+2}v_{p+3}\})\le \max\{\mu_5(G_{7,3,2,1}),\mu_1(P_{p-2}),\mu_1(P_{d-p-1})\}<4.
\]
If $p=2$, then $G'-v_{p+2}v_{p+3}\cong G_{7,3,2,1}\cup P_{d-p-1}$, so we have by Lemmas \ref{cauchy} and \ref{Hpq} that
\[
\mu_7(G')\le \mu_5(G'-v_{p+2}v_{p+3})=\max\{ \mu_5(G_{7,3,2,1}),\mu_1(P_{d-p-1})\}<4.
\]
If $p=d-1$, then $G'-v_{p-2}v_{p-1}\cong G_{7,3,2,1}\cup P_{p-2}$ and so by Lemmas \ref{cauchy} and \ref{Hpq},
\[
\mu_7(G')\le \mu_5(G'-v_{p-2}v_{p-1})=\max\{ \mu_5(G_{7,3,2,1}),\mu_1(P_{d-p-1})\}<4.
\]

\noindent
{\bf Case 3.} $r=q$.

In this case, $p=q=r$. By Lemma \ref{cauchy}, we assume that $uv,vw,uw\in E(G')$.
Let $u_i=v_i$ for $i=1,\dots,p-1$, $u_p=u$, $u_{p+1}=v_p$, $u_{p+2}=v$, $u_{i+2}=v_i$ for $i=p+1,\dots,d+1$ and $u_{d+4}=w$. Under this new labeling,
\[
G'-\{u_{p-1}u_{p+1}, u_{p-1}u_{p+2}, u_pu_{p+2}, u_pu_{p+3}, u_{p+1}u_{p+3}, u_{p+2}u_{d+4}, u_{p+3}u_{d+4}\}
\]
is a copy of $G_{d+4,d+2,p}$. So \[
L(G')=L(G_{d+4,d+2,p})+R,
\]
where $R=(r_{ij})_{(d+4)\times (d+4)}$  with \[
r_{ij}=\begin{cases}
2&\mbox{ if }i=j\in \{p-1,p,p+1,d+4\},\\
3&\mbox{ if }i=j\in \{p+2,p+3\},\\
-1&\mbox{ if }\{i,j\}\in \{\{p-1,p+1\},\{p-1,p+2\},\{p,p+2\},\{p,p+3\} \},\\
-1&\mbox{ if }\{i,j\}\in \{\{p+1,p+3\},\{p+2,d+4\},\{p+3,d+4\} \},\\
0&\mbox{ otherwise.}
\end{cases}
\]
As $R$ is permutational similar to $L(H\cup (d-2)K_1)$ where $H$ is a graph on $6$ vertices consisting of a cycle $u_{p-1}u_{p+1}u_{p+3}u_{p}u_{p+2}u_{p-1}$ and additional two edges $u_{p+2}u_{d+4}$ and $u_{p+3}u_{d+4}$, we have $\rho_6(R)=0$. So by Lemmas \ref{cw} and \ref{gndt}, we have \[
\mu_7(G')\le \mu_2(G_{d+4,d+2,p})+\rho_6(R)=4.
\]

Suppose that $\mu_7(G')=4$. By Lemma \ref{cw}, there exists a nonzero vector $\mathbf{x}$ such that $R\mathbf{x}=\mathbf{0}$ and $L(G_{d+4,d+2,p})\mathbf{x}=4\mathbf{x}$. As earlier, let $x_i=x_{u_i}$ for $i=1,\dots, d+4$.
From $R\mathbf{x}=\mathbf{0}$, we have $L(H)(x_{p-1}, x_{p+1},x_{p+2}, x_{p}, x_{p+3},x_{d+4})^\top=\mathbf{0}$, so
%
%
%
$x_{p-1}=\dots=x_{p+3}=x_{d+4}$.

Suppose first that $p\ge 3$.
From $L(G_{d+4,d+2,p})\mathbf{x}=4\mathbf{x}$ at $u_{p-1}$, we have
\[
3x_{p-1}-x_{p-2}-x_{p}-x_{d+4}=4x_{p-1},
\]
so $x_{p-2}=-3x_{p-1}$.
As $u_1\dots u_{p-1}$ is a pendant path of $G'$ at $u_{p-1}$, we have by Lemma \ref{cha} that  \[
x_i=(-1)^{i-1}(2i-1)x_1 \mbox{ for } i=1,\dots,p-1.
\]
Then \[
x_{p-2}=(-1)^{p-3}(2(p-2)-1)x_1=-3\cdot(-1)^{p-2}(2(p-1)-1)x_1,
\]
i.e., \[
(2p-5)x_1=3(2p-3)x_1.
\]
As  $2p-5\ne 3(2p-3)$, we have $x_1=0$, so $x_i=0$ for $i=1,\dots,p+3,d+4$. If $p=2$, this follows from
$L(G_{d+4,d+2,p})\mathbf{x}=4\mathbf{x}$ at $u_{p}$.

Now from $L(G_{d+4,d+2,p})\mathbf{x}=4\mathbf{x}$ at $u_i$ with $i=p+3,\dots, d+2$, we have $x_{i+1}=0$.
Thus  $\mathbf{x}=\mathbf{0}$, a contradiction.
Therefore, $\mu_7(G')<4$.
\end{proof}

The bound in Theorem \ref{x} can be improved under certain conditions.

\begin{theorem}\label{n-5}
Let $G$ be an $n$-vertex connected graph with diametral path $P:=v_1\dots v_{d+1}$, where $d\le n-5$.
If there exist at least three vertices outside $P$ for which no two have a common neighbor on $P$, then $m_G[n-d,n]\le n-d+1$.
\end{theorem}

\begin{proof}
As $d\le n-5$, there are at least four vertices outside $P$, say $w_1$, $w_2$, $w_3$ and $w_4$. Assume that no two of some  three vertices among  $w_1$, $w_2$, $w_3$ and $w_4$ have a common neighbor on $P$.

Let $p_i=\max\{j: v_j\in \Gamma_{G,P}(w_i)\}$ for $i=1,2,3,4$. Assume that $p_1\le \dots \le p_4$.
Let $H=G[V(P)\cup \{w_1,\dots,w_4\}]$.

Suppose first that $\Gamma_{G,P}(w_2)\cap \Gamma_{G,P}(w_3)=\emptyset$. Let $H-v_{p_{2}}v_{p_{2}+1}=H_1\cup H_2$, where
$V(H_1)=\{v_1, \dots, v_{p_2}, w_1,w_2\}$ and $V(H_2)=\{v_{p_2+1}, \dots, v_{d+1}, w_3,w_4\}$. Evidently,  $H_1$ ($H_2$, respectively) is a connected graph of order $p_2+2$ ($d-p_2+3$, respectively) with diameter $p_2-1$ ($d-p_2$, respectively).
Since  no two of three vertices outside $P$ have a common neighbor on $P$ in $G$,   there are two possibilities:

(i) $H_1\not\cong G_{p_2+2,p_2-1,t}$ for any $2\le t\le p_2-1$, and $H_1\not\cong G_{p_2+2,p_2-1,r,1}$
for any  $2\le r\le p_2-1$. By Lemma \ref{z+}, we have $\mu_3(H_1)<5$. By  Theorem \ref{x1}, $\mu_4(H_2)<5$.
By Lemma \ref{cauchy}, one gets
\[
\mu_7(H)\le \mu_6(H_1\cup H_2)\le \max\{\mu_3(H_1),\mu_4(H_2) \}<5.
\]

(ii) $H_2\not\cong G_{d-p_2+3,d-p_2,t}$ for any $2\le t\le d-p_2$, and $H_2\not\cong G_{d-p_2+3,d-p_2,r,1}$
for any  $2\le r\le d-p_2$. By Theorem \ref{x1}, we have $\mu_4(H_1)<5$.
By Lemma \ref{z+}, $\mu_3(H_2)<5$.
By Lemma \ref{cauchy}, one gets
\[
\mu_7(H)\le \mu_6(H_1\cup H_2)\le \max\{\mu_4(H_1),\mu_3(H_2) \}<5.
\]

Suppose next that $\Gamma_{G,P}(w_2)\cap \Gamma_{G,P}(w_3)\ne \emptyset$.
By the assumption, $\Gamma_{G,P}(w_1)\cap \Gamma_{G,P}(w_2)=\emptyset$ or $\Gamma_{G,P}(w_3)\cap \Gamma_{G,P}(w_4)=\emptyset$, say $\Gamma_{G,P}(w_3)\cap \Gamma_{G,P}(w_4)=\emptyset$.
Let $H-v_{p_3}v_{p_3+1}=H_3\cup H_4$. Then $H_3$ ($H_4$, respectively) is a connected graph of order $p_3+3$ ($d-p_3+2$, respectively) with diameter $p_3-1$ ($d-p_3$, respectively). By Theorem \ref{x2}, $\mu_6(H_3)<5$.  Lemma \ref{mu1}, we have  $\mu_1(H_4)<5$.
Now,  by Lemma \ref{cauchy}, one gets
\[
\mu_7(H)\le \mu_6(H-v_{p_3}v_{p_3+1})\le \max\{\mu_6(H_3),\mu_1(H_4) \}<5.
\]

Therefore, $\mu_7(H)<7$ in each case. Let
$B$ be the principal submatrix of $L(G)$ corresponding to vertices of $H$ and $M$ is the diagonal matrix whose diagonal entry corresponding to vertex $z$ is $\delta_G(z)-\delta_H(z)$ for $z\in V(H)$. Then,
by Lemma \ref{interlacing} and \ref{cw}, \[
\mu_{n-d+2}(G)=\rho_{n-(d+5)+7}(L(G))\le  \rho_7(B)\le \mu_7(H)+\rho_1(M)<n-d,
\]
as desired.
\end{proof}

\section{Proof of Theorem \ref{d-2}}

Theorem \ref{d-2} follows from Theorems \ref{2d-2} and \ref{2d-1}.

\begin{theorem}\label{2d-2}
Let $G$ be a connected graph of order $n$ with diameter $d$. If $2\le d\le \lfloor\frac{n+3}{2} \rfloor$, then $m_G[n-2d+4,n]\le n-2$.
\end{theorem}

\begin{proof}
It suffices to show that $\mu_{n-1}(G)< n-2d+4$.
If $d=2$, then $G$ is a spanning subgraph of $K_n-e$ for some $e\in E(K_n)$,  so we have by Lemma \ref{cauchy} that $\mu_{n-1}(G)\le n-2<n=n-2d+4$, as desired.
Suppose that $d\ge 3$.
For $i=2,\dots, d-1$, let $V_i$ be the set of vertices of $G$ such that the distance to $v_1$ is $i-1$.
Let $V_d$ be the set of vertices of $G$ such that the distance to $v_1$ is $d-1$ and the neighbors of $v_{d+1}$.
Evidently, $v_i\in V_i$ and  $V_i$ is a cut set of $G$ for each $i=2,\dots,d$.
As $P$ is a diametral path, $V_i\cap V_j=\emptyset$ if $i\ne j$ and there is no edge between $V_i$ and $V_j$ if $|j-i|\ge 2$.
If $\kappa(G)\ge n-2d+5$, then
\[
2+(n-2d+5)(d-2)\le |\{v_1,v_{d+1}\}|+\sum_{i=2}^d|V_i|\le n,
\]
i.e., $2d^2-(n+9)d +3n+8\ge 0$, so $d<3$ or $d>\frac{n+3}{2}$, a contradiction.
So  $\kappa(G)\le n-2d+4$.
As $d\ge 3$,  $G$ is not a join, so we have by Lemma \ref{mun-1} that $\mu_{n-1}(G)<\kappa(G)\le n-2d+4$.
\end{proof}

Evidently, $m_{K_n-e}[n,n]=n-2$.  Let  $R_1$ ($R_2$, respectively) be the graph on $8$ vertices ($7$ vertices, respectively) with diameter $5$ in Fig. \ref{d5}. By a direct calculation, we have $\mu_6(R_1)=2$ and $\mu_5(R_2)=1$. By Theorem \ref{2d-2}, $m_{R_1}[2,8]=6$ and $m_{R_2}[1,7]=5$, agreeing the bound in Theorem \ref{2d-2} for  $d=\frac{n+2}{2}=5$ and $d=\frac{n+3}{2}=5$, respectively.

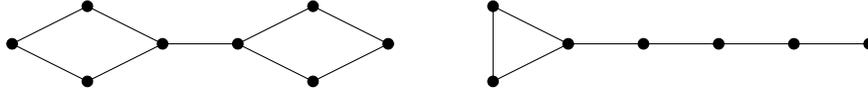
\begin{figure}[htbp]
\centering
\begin{tikzpicture}
\draw [black] (8,0.5)--(9,0);
\draw[black] (8,0.5)--(9,1);
\draw[black] (9,0)--(10,0.5);
\draw[black] (9,1)--(10,0.5);
\draw[black] (10,0.5)--(11,0.5);
\draw[black] (11,0.5)--(12,0);
\draw[black] (11,0.5)--(12,1);
\draw[black] (12,0)--(13,0.5);
\draw[black] (12,1)--(13,0.5);
\filldraw[black] (8,0.5) circle (2pt);
\filldraw[black] (9,0) circle (2pt);
\filldraw[black] (9,1) circle (2pt);
\filldraw[black] (10,0.5) circle (2pt);
\filldraw[black] (11,0.5) circle (2pt);
\filldraw[black] (12,0) circle (2pt);
\filldraw[black] (12,1) circle (2pt);
\filldraw[black] (13,0.5) circle (2pt);
\end{tikzpicture} \ \ \ \ \ \   \ \
\begin{tikzpicture}
\draw  [black](1,0.5)--(5,0.5);
\draw [black] (0,0)--(0,1);
\draw[black] (0,0)--(1,0.5);
\draw[black] (0,1)--(1,0.5);
\filldraw [black] (0,0) circle (2pt);
\filldraw[black] (0,1) circle (2pt);
\filldraw [black] (1,0.5) circle (2pt);
\filldraw [black] (2,0.5) circle (2pt);
\filldraw [black] (3,0.5) circle (2pt);
\filldraw [black] (4,0.5) circle (2pt);
\filldraw[black] (5,0.5) circle (2pt);

\end{tikzpicture}
\caption{The graph $R_1$ (left) and $R_2$ (right).}
\label{d5}
\end{figure}

If $3\le d\le \lfloor\frac{n+1}{2}\rfloor$, Theorem \ref{2d-2} may be improved as follows.

\begin{theorem}\label{2d-1}
Let $G$ be a connected graph of order $n$ with diameter $d$. If $3\le d\le \lfloor \frac{n+1}{2}\rfloor$, then $m_G[n-2d+4,n]\le n-3$.
\end{theorem}

\begin{proof} The case for $d=3$ is known from \cite[Theorem 6]{XZ}, and the case for $d=4$ follows from Theorem \ref{x} (i). Suppose in the following that $d\ge 5$. It suffices to show that $\mu_{n-2}(G)<n-2d+4$.

Let $P:=v_1\dots v_{d+1}$ be a diametral path of $G$.
For $i=1,\dots, d-1$, let $V_i$ be the set of vertices of $G$ such that the distance to $v_1$ is $i-1$.
Let $V_{d}$ be the set of vertices of $G$ except $v_{d+1}$ such that the distance to $v_1$ is $d-1$ or $d$. Let $V_{d+1}=\{v_{d+1}\}$.
By Lemma \ref{cauchy}, we assume that $G[V_i\cup V_{i+1}]$ is complete for each $i=2,\dots, d$.
Note that $V_i$ is a cut set of $G$ for each $i=2,\dots,d$ and that  $v$ is a cut vertex of  $G$
if and only if $v=v_i$ and $|V_i|=1$ for some $i=2,\dots, d$.

Let $s$ be the number of sets $V_2,\dots, V_d$ with cardinality $1$.
We divide the proof into two cases.

\noindent
{\bf Case 1.} $\delta_{n-1}(G)\ge n-2d+4$.

Note that  $\max\{|V_2|, |V_d|\}=\max\{\delta_G(v_1), \delta_G(v_{d+1})\}\ge \delta_{n-1}(G)\ge n-2d+4$.
Assume that $|V_2|\ge n-2d+4>2$ and $|V_j|=\min\{|V_i|: i=3,\dots, d\}$. Then
 $|V_i|\ge 2$ for $i=3,\dots,d$ if $s=0$, and $|V_i|\ge 2$ for $i=3,\dots, d$ with $i\ne j$ if $s=1$. Thus, if $s=0,1$, then
\[
n+1=2+1+(n-2d+4)+2(d-3)\le |\{v_1,v_{d+1}\}|+|V_j|+|V_2|+\sum_{i=3\atop i\ne j}^d|V_i|\le n,
\]
a contradiction. So $s\ge 2$. Assume that  $|V_{\ell}|=1$. Then $v_\ell$ is one cut vertex of $G$. Suppose
that there is a component $G_0$ of $G-v_\ell$ such that $G_0$ has a cut vertex. Then $\kappa(G_0)=1$ and by Lemma \ref{mun-1}, $\mu_{|V(G_0)|-1}(G_0)\le \kappa(G_0)=1$.
Let $B$ be the principal submatrix of $L(G)$ by deleting the row and column corresponding to vertex $v_\ell$.
By Lemma \ref{cw}, $\rho_{n-3}(B)\le \mu_{n-3}(G-v_\ell)+\rho_1(B-L(G-v_\ell))=\mu_{n-3}(G-v_\ell)+1$.
Then, by Lemma \ref{interlacing}, we have
\[
\mu_{n-2}(G)\le \rho_{n-3}(B)\le \mu_{n-3}(G-v_\ell)+1\le \mu_{|V(G_0)|-1}(G_0)+1\le 2<n-2d+4,
\]
as desired.
Suppose that there is no cut vertices of any component of $G-v_\ell$. Then $s=2$ and either $|V_{\ell-1}|=1$ or $|V_{\ell+1}|=1$, say $|V_{\ell+1}|=1$.
As
\[
n=4+n-2d+4+2(d-4)\le |\{v_1,v_{d+1},v_\ell,v_{\ell+1} \}|+|V_2|+\sum_{i=3\atop i\ne \ell, \ell+1}^d|V_i|\le n,
\]
we have $|V_2|=n-2d+4$ and $|V_i|=2$ for $i=3,\dots, d$ with $i\ne \ell, \ell+1$, where  $3\le\ell\le d-1$.
If $\ell=d-1$, then $\delta_G(v_{d+1})=1$ and $\delta_G(v_{n-1})=2$, so $n-2d+4\le \delta_{n-1}(G)\le 2$, which is a contradiction.
So $\ell\le d-2$ and  $|V_d|= 2$.
Let $B'$ be the principal submatrix of $L(G)$ by deleting the rows and columns corresponding to vertices in $V_d$. Let  $G_1=G-V_d-v_{d+1}$. By Lemma \ref{cw},  $\rho_{n-4}(B')\le \mu_{n-4}(G-V_d)+\rho_1(B'-L(G-V_d))=\mu_{n-4}(G_1)+2$. Note that $G_1$ is not a join with a cut vertex $v_\ell$.  By Lemma \ref{mun-1}, $\mu_{n-4}(G_1)<\kappa(G_1)=1$.
Therefore, by Lemma \ref{interlacing},
\[
\mu_{n-2}(G)\le \rho_{n-4}(B')\le \mu_{n-4}(G_1)+2< \kappa(G_1)+2=3\le n-2d+4,
\]
as desired.

\noindent
{\bf Case 2.} $\delta_{n-1}(G)\le n-2d+3$.

By Corollary \ref{delta}, $\mu_{n-2}(G)\le \delta_{n-1}(G)+1\le n-2d+4$. Suppose by contradiction that
$\mu_{n-2}(G)=n-2d+4$.
Then $\mu_{n-2}(G)=\delta_{n-1}(G)+1$ and $\delta_{n-1}(G)=n-2d+3$.
Let $u_1$ and $u_2$ be two vertices of degree $\delta_n(G)$ and $\delta_{n-1}(G)$ in $G$, respectively.
By Corollary \ref{delta} and the fact that $\mu_{n-2}(G)=\delta_{n-1}(G)+1$, we have the following two cases.

\noindent
{\bf Case 2.1.} $u_1u_2\notin E(G)$, $\delta_{n-1}(G)=\delta_n(G)=\frac{n-2}{2}$ and $N_G(u_1)\cap N_G(u_2)=\emptyset$.

Note that $V(G)=\{u_1,u_2\}\cup N_G(u_1)\cup N_G(u_2)$. Let $U_i=N_G(u_i)$ for $i=1,2$.
As $G$ is connected, there is a vertex $w_i\in U_i$ with $i=1,2$ such that $w_1w_2\in E(G)$. The distance between any vertex pair of vertices in $\{u_1,u_2\}\cup U_i$ with $i=1,2$ is at most three.
Let $z_1\in U_1\setminus \{w_1\}$. If $z_1w_1\in E(G)$, then the distance between $z_1$ and any vertex in $U_2$ is at most three. If $z_1w_1\notin E(G)$, then as $\delta_G(z_1)\ge \delta_n(G)=\frac{n-2}{2}=|U_1|$, we have $z_1z_2\in E(G)$ for some $z_2\in U_2$, so the distance between $z_1$ and any vertex in $U_2$ is at most three. This shows that $d\le 3$, a contradiction.

\noindent
{\bf Case 2.2.} $u_1u_2\in E(G)$ and $N_G(u_1)\setminus\{u_2\}=N_G(u_2)\setminus\{u_1\}$.

Note that $\delta_n(G)=\delta_G(u_1)=\delta_G(u_2)=\delta_{n-1}(G)=n-2d+3$. Then  $|V_2|,|V_d|\ge \delta_n(G)=n-2d+3\ge 2$.  Let $|V_j|=\min\{|V_i|: i=3,\dots, d-1\}$. If $|V_j|\ge 2$, then
\[
2+(n-2d+3)\cdot 2+2(d-3)\le |\{v_1,v_{d+1}\}|+|V_2|+|V_d|+\sum_{i=3}^{d-1}|V_i|\le n,
\]
i.e., $n\le 2d-2$, which is a contradiction.
So $|V_{\ell}|=1$ for some $\ell$ with $3\le \ell\le d-1$.
Denote by $B$ the principal submatrix of $L(G)$ by deleting the row and column corresponding to vertex  $v_\ell$.

Suppose that there is a component $G_0$ of $G-v_\ell$ such that $\kappa(G_0)=1$.
It then follows from Lemmas \ref{interlacing}, \ref{cw} and \ref{mun-1} that
\[
\mu_{n-2}(G)\le \rho_{n-3}(B)\le \mu_{n-3}(G-v_\ell)+1\le \mu_{|V(G_0)|-1}(G_0)+1\le \kappa(G_0)+1=2<n-2d+4,
\]
a contradiction.
So there is no cut vertices of any component of $G-v_\ell$,  $s=1,2$, and if $s=2$, then one of $v_{\ell-1}$ and $v_{\ell+1}$, say $v_{\ell+1}$, is a cut vertex of $G$.
Thus,  $G-v_{\ell}$  consists of two components, say $H$ and $F$, with  $v_1,\dots, v_{\ell-1}\in V(H)$ and  and $v_{\ell+1},\dots, v_{d+1}\in V(F)$.

If $H$ and $F$ are both complete, then $d\le 4$, which is a contradiction to the assumption that $d\ge 5$. Assume that $H$ is not complete. Let $p=|V(H)|$.

If  $F$ is not complete, then as one of $u_1$ and $u_2$ lies in $G-v_{\ell}=H\cup F$ and $\delta_G(u_1)=\delta_G(u_2)=n-2d+3$, we have $\min\{\delta_p(H), \delta_{n-p}(F)\}\le n-2d+3$, so we assume that $\delta_{p}(H)\le n-2d+3$ (if $\delta_{n-p}(F)\le n-2d+3$, then we exchange the roles of $H$ and $F$). If $F$ is complete, then  $\delta_p(H)\le n-2d+3$, as otherwise, we have  $|V_d|\ge 2$, $s=1$,  $\ell=d-1$, and then
\[
3+(n-2d+4)+2(d-4)+(n-2d+3)\le |\{v_1,v_{d+1},v_{d-1}\}|+\sum_{i=2}^d|V_i|\le n,
\]
i.e., $n\le 2d-2$, which is a contradiction.
It then follows that $\kappa(H)\le \delta_{p}(H)\le n-2d+3$.
By Lemma \ref{mun-1}, $\mu_{p-1}(H)\le \kappa(H)\le n-2d+3$.
Now,  by Lemmas \ref{interlacing} and \ref{cw}, we have
\[
n-2d+4=\mu_{n-2}(G)\le \rho_{n-3}(B)\le \mu_{n-3}(G-v)+1\le \mu_{p-1}(H)+1\le n-2d+4,
\]
so $\mu_{p-1}(H)=n-2d+3=\kappa(H)=\delta_p(H)$. By Lemma \ref{mun-1}, $H$ is a join, say $H=H_1\vee H_2$, and one of $H_1$ and $H_2$, say $H_1$,  is disconnected and the other $H_2$ has order $n-2d+3$, so
$\{v_1\}\cup V_3\subseteq V(H_1)$ and $\ell=4$.

Suppose that $s=1$.
Then we have
\[
3+(n-2d+3)+2(d-4)+(n-2d+3)\le |\{v_1,v_{d+1},v_4 \}|+\sum_{i=2\atop i\ne 4}^d|V_i|\le n,
\]
i.e., $n\le 2d-1$,  so $n=2d-1$ and $|V_i|=2$ for $i=2,\dots, d$ with $i\ne 4$.
This is impossible because there are no vertices $u_1$ and $u_2$ such that
$u_1u_2\in E(G)$ and $\delta_G(u_1)=\delta_G(u_2)=2$.

Suppose that $s=2$. Then $4=\ell\le d-2$.
Then $H'=G[V(H)\cup\{v_4\}]$ is a component of $G-v_5$ and it is not a join. Note that $\delta_{p+1}(H')\le \delta_p(H)\le n-2d+3$. So $\kappa(H')\le n-2d+3$. By Lemmas \ref{interlacing}, \ref{cw} and \ref{mun-1}, we have
\[
n-2d+4=\mu_{n-2}(G)\le \rho_{n-3}(B')\le \mu_{n-3}(G-v_5)+1\le \mu_{p}(H')+1<\kappa(H')+1\le n-2d+4,
\]
a contradiction.
\end{proof}

\section{Concluding remarks}

As mentioned in Section 1 by excluding the trivial cases, we propose the following conjecture, which is true for $c=0,1,2, d-3, d-2$.

\begin{conjecture} \label{you} Let $G$ be a connected graph of order $n$  with diameter $d\ge 2$. If
$c=0,\dots, d-2$ with  $\max\{2,c\}\le d\le n-2-c$, then $m_G[n-d+2-c,n]\le n-d+c$.
\end{conjecture}

Note that in Conjecture \ref{you}, as the interval  $[\max\{2,c\}, n-2-c]$ becomes smaller, the bound for the number of Laplacian eigenvalues in $[n-d+2-c,n]$ becomes larger.
We may go further to prove Conjecture \ref{you} for $c=3,d-4$ with more detailed analysis. However, for the general $c$, it seems that some different technique is needed. Anyway, it is helpful to understand how the Laplacian eigenvalues are distributed and how this distribution   is related to the diameter.

\bigskip

\noindent {\bf Declaration of competing interest}\\

There is no competing interest.

\bigskip

\noindent {\bf Data availability}\\

No data was used for the research described in the article.

\bigskip
\noindent {\bf Acknowledgements}\\

The authors thank  Professors Domingos M. Cardoso, Vilmar Trevisan, Saieed Akbari, Elismar R. Oliveira and Stephen T. Hedetniemi for helpful comments and suggestions on an early version of this paper.
This work was supported by the National Natural Science Foundation of China (No.~12071158).

\end{document}